\documentclass[preprint,12pt]{elsarticle}

\usepackage{graphicx,xcolor}
\usepackage[pagebackref=true,colorlinks,linkcolor=red,citecolor=blue]{hyperref}
\usepackage{multirow}
\usepackage{adjustbox}
\usepackage{natbib}

\def\newpic#1{}
\def\mtwo#1#2{\raise .8ex\hbox{
     ${#1_{{\displaystyle #2}}}$}}
 \usepackage{sansmath}
 \usepackage{amsmath}
\DeclareSymbolFont{extraup}{U}{zavm}{m}{n}
\DeclareMathSymbol{\varheart}{\mathalpha}{extraup}{86}
\DeclareMathSymbol{\vardiamond}{\mathalpha}{extraup}{87}
\DeclareMathSymbol{\varclub}{\mathalpha}{extraup}{88}
\DeclareMathSymbol{\varspade}{\mathalpha}{extraup}{85}
\newtheorem{theorem}{Theorem}

\newtheorem{conjecture}{Conjecture}

\newtheorem{proposition}{Proposition}
\newtheorem{corollary}{Corollary}
\newtheorem{definition}{Definition\!}

\newenvironment{proof}{{\bf Proof.}}{\hspace*{1mm}\hfill\rule{2mm}{2mm}}

\newtheorem{pretheorema}{{\bf Theorem}}

\newtheorem{prelemmab}{{\bf Lemma}}

\newtheorem{prepropositiona}{{\bf Proposition}}

\input amssym.tex
\def\n#1{\vbox to 3mm{\vspace{1mm}\vfill \hbox to 2.0mm{\hfill
$#1$\hfill} \vfill }}
\def\m#1#2{\raise 0.2ex\hbox{
${#1_{  #2}}$}}
\def\x#1{\raise 0.5ex\hbox{
${#1}$}}
\def\Cchi{{\raisebox{.2ex}{\large $\chi$}}}
\journal{Journal of Discrete Mathematics}

\begin{document}

\begin{frontmatter}

\title{On the chromatic number of Latin square graphs}


\author[label1]{Nazli Besharati}
\address[label1]{Department of Mathematical Sciences,  Payame Noor University,
P.O. Box 19395-3697,   Tehran, I. R. Iran}
\author[label2]{Luis Goddyn}
\address[label2]{Department of Mathematics, Simon Fraser University, Burnaby, BC V5A 1S6 Canada}
\author[label3]{ E.S. Mahmoodian}
\author[label3]{M. Mortezaeefar}
\address[label3]{Department of Mathematical Sciences, Sharif University of
Technology, P. O. Box 11155-9415, Tehran, I. R. Iran}


\begin{abstract}


The {\sf chromatic number} of a Latin square is the least number of partial transversals which cover its cells.
This is just the chromatic number of its associated {\sf Latin square graph}.
Although Latin square graphs have been widely studied as strongly regular graphs,
their chromatic numbers appear to be unexplored.
We determine the chromatic number of a circulant Latin square,
and find bounds for some other classes of Latin squares.
With a computer, we find the chromatic number for all main classes of Latin squares of order at most  eight.
\end{abstract}

\begin{keyword} Transversal, Partial transversal, Latin square, Latin square graphs, Net graphs, Coloring,  Cayley table, Orthogonal array, Orthogonal Mate.
 \MSC  05B15- 05C15




\end{keyword}

\end{frontmatter}





%
%
%
%
%
%
%
%
%

\section{Introduction and preliminaries}

Let $\mathbf L$ be a Latin square of order $n$. The \textsf{chromatic number} of  $\mathbf L$, denoted by $\Cchi(\mathbf{L})$,
is the minimum number of partial transversals of $\mathbf L$ which together cover the cells of $\mathbf{L}$.
Since each partial transversal uses at most $n$ of the $n^2$ cells in $\mathbf L$, we observe the following.
\begin{proposition}
Every Latin square $\mathbf L$ of order $n$ satisfies $\Cchi(\mathbf L)\ge n$,
with equality holding if and only if $\mathbf L$ has an orthogonal mate.
\end{proposition}
Therefore  $\Cchi(\mathbf L)$ serves as a measure of how close $\mathbf L$ is to having an orthogonal mate.
We are surprised that reference to this natural invariant seems to be absent from the  substantial literature regarding transversals and orthogonality of Latin squares\footnote{See the Addendum for further details}.
An early version of some of our results appears in the M.\ Sc.\ thesis \cite[in Farsi, Persian language]{MortezaeefarThesis} of the fourth author, under the supervision of the third author.
%
%
In this paper, we find some bounds on $\Cchi(\mathbf L)$ for general Latin squares and special classes such as complete Latin squares, Cayley tables of groups, circulants, and all Latin squares of order at most eight.

For the definitions not given here one may refer to~\cite{MR2368647} and~\cite{MR2246267}.
Let ${\bf L}$ be a Latin square of order $n$ with \textsf{cells} $\{ (r,c) \mid r,c \in \{0,1,2,\ldots, n-1\} \}$;
each cell contains a \textsf{symbol} from an alphabet of size $n$, and no row or column of $\textbf{L}$ contains a repeated symbol.
A cell $(r,c)$ containing the symbol $s = \mathbf L_{r,c} $ is sometimes represented by the triple $(r,c,s)$.
A {\sf partial transversal of length $k$} is a set of $k$ cells, where no two cells have the same row, column or symbol.
A {\sf transversal} is a partial transversal of length~$n$.
The Latin square $\mathbf L$ has an \textsf{orthogonal mate} if and only if it has a decomposition into disjoint transversals.  
We say that $\textbf L$ is \textsf{row-complete}
  if every ordered pair of distinct symbols appears (exactly once) in the set
  \begin{align*}
 			&\{ (s, s') \mid (r,c,s), (r,c+1,s') \in \mathbf L,\text{ for $0 \le r \le n-1$ and $0 \le c \le n-2$} \} .
 \end{align*}
 %
 The \textsf{Latin square graph} of $\mathbf L$ is the simple graph $\Gamma(\mathbf L)$ whose vertices are the cells of~$\mathbf{L}$,
 and where distinct cells $(r,c,s)$ and $(r',c',s')$ are adjacent if (exactly) one of the equations $r=r'$, $c=c'$, $s=s'$ is satisfied.
 Accordingly, each edge of  $\Gamma(\mathbf L)$ is called, respectively, a \textsf{row edge}, a  \textsf{column edge} or a \textsf{symbol edge}.
Latin square graphs were introduced by R. C. Bose~\cite{MR0157909} as examples of strongly regular graphs;
see \cite[Section 10.4]{MR1829620} for further discussion.
Bose used the notation $L_3(n)$ for this graph.
 But this notation does not specify the Latin square from which the graph arises. So we use the notation $\Gamma({\bf L})$ for the graph corresponding to the given Latin square ${\bf L}$.
The independent sets of $\Gamma(\mathbf L)$ are the partial transversals of $\mathbf L$,
 and $\Cchi(\mathbf L)$ is the chromatic number of $\Gamma(\mathbf L)$.
 The isomorphism class of $\Gamma(\textbf L)$ is not affected by relabelling the rows, columns or symbols of $\mathbf L$,
 nor is it changed by applying a fixed permutation to the coordinates of every triple $(r,c,s)$ in $\mathbf L$.
 Thus $\Cchi(\mathbf L)$ is an invariant of the \textsf{main class} of $\mathbf L$.

 Let $(G,\circ)$ be a finite group of order $n$.
A \textsf{Cayley table} for $G$ is an $n \times n$ matrix, denoted $L_G$, where the cell $(i,j)$ contains the group element $g_i \circ g_j$,
for some fixed enumeration $ G = \{g_0,\dots , g_{n-1}\} $.
It is easy to see that $L_G$ is a Latin square.
If $G$ is a cyclic group, then $L_G$ is called a \textsf{circulant} Latin square.
Figure \ref{fig: LatinSquare3} shows the graph of the circulant $L_{{\Bbb Z}_3}$.
\begin{figure}[h]
\begin{center}
\includegraphics[scale=0.6]{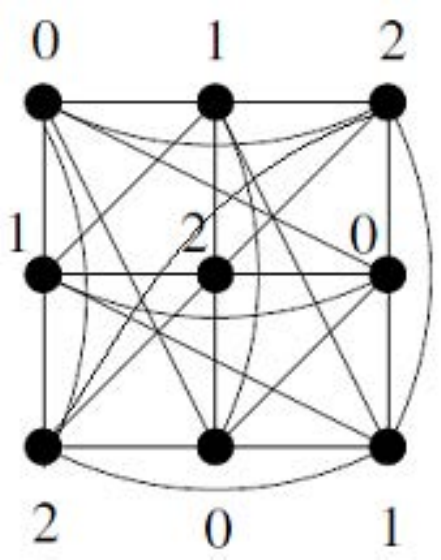}
\caption{The graph {$\Gamma({L_{{\Bbb Z}_3})}$ with each vertex $(r,c,s)$ labelled by $s$.}}
\label{fig: LatinSquare3}
\end{center}
\end{figure}

We summarize the results of this paper.
Let $\mathbf L$ be a Latin square of order $n$.
\begin{itemize}
\item
$n \le \Cchi(\mathbf L) \le 3n-2$.

\item
If $\mathbf L$ is row-complete, then $ \Cchi(\mathbf L) \le 2n$.

\item For large $n$ we have $\Cchi(\mathbf L) = n + o(n)$.

\item
For every  group $G$ of order $n$,  either $\Cchi(L_G) = n$ or $\Cchi(L_G) \ge n+2$.
\item
If $\mathbf L$ is a circulant (that is $\mathbf L \cong L_{{\Bbb Z}_n}$),
then $\Cchi(\mathbf L) = \begin{cases} n & \text{ if $n$ is odd}\\ n+2 & \text{ if $n$ is even.}\end{cases}$

\item
If $n\le8$, then
\begin{equation}\label{ineq}
 \Cchi(\mathbf L) \le \begin{cases} n+1 & \text{ if $n$ is odd}\\ n+2 & \text{ if $n$ is even.}\end{cases}
\end{equation}
\end{itemize}
%
%
We propose the following.
%
%
\begin{conjecture}\label{ques:main}
Every  Latin square $\mathbf L$ satisfies \eqref{ineq}.
\end{conjecture}

%
Conjecture \ref{ques:main} would surely be challenging to prove, even for Cayley tables of groups.
Since $(n+1)(n-1)<n^2$ and $(n+2)(n-2)<n^2$,
every Latin square $\mathbf L$ which satisfies \eqref{ineq} must have a transversal (if $n$ is odd) or a partial transversal of length $n-1$ (if $n$ is even).
%
So Conjecture \ref{ques:main} would imply two long-standing  conjectures of Brualdi-Stein
 and  Ryser.
\begin{conjecture}
\label{Brualdi-Stein}
 {\rm (\cite{MR0351850,MR0387083})}
Every Latin square of even order $n$ contains a partial
transversal of length  $n -1$.
\end{conjecture}
\begin{conjecture}
\label{Ryser}
 {\rm (\cite{hjryser})}
Every Latin square of odd order contains a transversal.
\end{conjecture}

\section{Some Upper Bounds}

%

The graph of a Latin square $\mathbf L$ of order $n$ is regular of degree $3n-3$.
This immediately gives $\Cchi(\mathbf L) \le 3n-2$.
We can improve this bound in case $\mathbf L$ has additional structure.

A {\sf $k$-plex}  is a set of $kn$ cells which has $k$ representatives from each row and each column and each symbol of \textbf{L}. A {\sf $(k_1, k_2,\dots , k_d)$-partition} is a partition $K_1, K_2,\dots,K_d$ where each $K_i$ is a $k_i$-plex.


\begin{proposition}
If a Latin square ${\bf L}$ of order $n$ has a  $(k_1, k_2,\dots , k_d)$-partition,  then
$\Cchi({\bf L}) \leq 3n-2d.$
\end{proposition}
\begin{proof}{
Let
$K_1, K_2,\dots,K_d$ be a $(k_1, k_2,\dots , k_d)$-partition of $\mathbf L$.
For each $i$, the induced subgraph $\Gamma(\mathbf L)[K_i]$ is regular of degree $3k_i-3$,
so it is $(3k_i-2)$-colorable.
Thus, $\Cchi({\bf L}) \leq \displaystyle\sum_{i=1}^d (3k_i-2)= 3n-2d.$
}\end{proof}

Wanless~\cite{MR1912794, MR2866738} has conjectured that
every Latin square of order $n$ has a $(k_1, k_2,\dots , k_d)$-partition where $d \ge \lfloor \frac{n}{2}\rfloor$.
If true, his conjecture would improve the general upper bound to $\Cchi(\mathbf L) \le 2n+1$.
\begin{corollary}
Let $t$ be the maximum number of disjoint transversals in a Latin square ${\bf L}$ of order $n$. Then
$$
    \left\lbrace
        \begin{array}{l l}
             \Cchi({\bf L}) \leq  3n-2t-2 & \qquad if \ t \leq n-2 \quad   \\
             \Cchi({\bf L}) =  n          & \qquad otherwise.
        \end{array}
            \right.
$$
\end{corollary}
%

A Latin square ${\bf L}$ of order $n$  is  {\sf row-complete}  (sometimes it is called a Roman square)
 if the ordered pairs $(\mathbf L_{i,j}, \mathbf L_{i,j+1})$ are all
  distinct for $ 0 \leq i \leq n-1$ and $ 0 \leq  j  \leq n - 2$.
Row-complete Latin squares are used
in the design of sequential
experiments \cite{MR0351850}.

\begin{proposition}
\label{Row complete latin}
If ${\bf L}$ is a row-complete
 Latin square of order $n$, then $\Cchi(\mathbf L) \leq 2n$.
\end{proposition}
\begin{proof}
We color each cell $\mathbf L_{i,j}$,  $ 0 \leq i \leq n-1$ and $ 0 \leq  j  \leq n - 2$,  with the entry of $\mathbf L_{i,j+1}$.
By the definition of row-complete, the cells with color $c$ are a  partial transversal  in ${\bf L}$.
We color the last column of ${\bf L}$
with $n$ new colors.
This gives $\Cchi(\mathbf L) \leq 2n$.
See Figure \ref{fig:ls_complete} for an example.
\end{proof}
%
%

\begin{figure}[h]
\begin{center}
\def\arraystretch{1.0}
\begin{tabular}
{|@{\hspace{2pt}}c@{\hspace{1.5pt}}|@{\hspace{1.5pt}}c@{\hspace{1.5pt}}|@{\hspace{1.5pt}}c@{\hspace{1.5pt}}|@{\hspace{1.5pt}}c@{\hspace{1.5pt}}|}\hline
\m{0}{1}   &\m{1}{3}     &\m{3}{2}     & \m{2}{\textcolor{red}4} \\ \hline
\m{1}{2}   &\m{2}{0}     &\m{0}{3}     & \m{3}{\textcolor{red}5} \\ \hline
\m{2}{3}   &\m{3}{1}     &\m{1}{0}     & \m{0}{\textcolor{red}6} \\\hline
\m{3}{0}   &\m{0}{2}     &\m{2}{1}     & \m{1}{\textcolor{red}7}  \\\hline
\end{tabular}
\caption{Subscripts indicate an $8$-coloring of row-complete Latin square of order $4$.}
\label{fig:ls_complete}
\end{center}
\end{figure}


The above bounds are far from optimal for large Latin squares.
Let $V$ be the disjoint union $R \cup C \cup S$ where $R$, $C$, $S$ is the set of rows, columns, and symbols of a Latin square $\mathbf L$ of order $n$.
Let $\mathcal H = (V,E)$ be the $3$-uniform $3$-partite hypergraph, where there is a hyperedge $\{r,c,s\} \in E$ for every cell $(r,c,s) \in \mathbf L$.
Then $\mathcal H$ is a {\sf linear} hypergraph i.e., no two hyperedges share more than one vertex, and $\mathcal H$ is regular of degree $n$.
Also $\Cchi(\mathbf L)$ is the {\sf chromatic index} of $\mathcal H$, the least number of colors needed to color $E$ so that no two  adjacent hyperedges get the same color.  The following general result of Molloy and Reed \cite{MR1801140} implies an asymptotically optimal bound for~$\Cchi(\mathbf L)$.
\begin{theorem}
For every $k$ there exists a constant $c_k$ such that every $k$-uniform linear hypergraph of maximum degree $n$ has chromatic index at most
$n + c_k(\log n)^4 n^{1-1/k}$.
\end{theorem}

\begin{corollary}
As $n \to \infty$ every Latin square $\mathbf L$ of order $n$ satisfies $\Cchi(\mathbf L) \le n + o(n)$.
\end{corollary}
\section{Chromatic Number of Cayley Tables}

The set of finite groups $G$ for which $\Cchi(L_G) = |G|$ has been recently characterized.
\begin{theorem} \label{thm: admissible groups}
For any finite group $G$ of order $n$
with identity element $\epsilon$,
the following are equivalent.
\begin{enumerate}
\item $\Cchi(L_G) = n$
\item $\Cchi(L_G) \le n+1$.
\item $L_G$ has a transversal.
\item For some enumeration $g_1, g_2, \dots, g_n$ of $G$ we have $g_1 g_2 \dots g_n = \epsilon$.
\item Every Sylow 2-subgroup of G is  either trivial or non-cyclic.
\end{enumerate}
\end{theorem}
\begin{proof}
The equivalence of statements 1.\ and 3.\ is classic, while the equivalence of 4.\ and 5.\ was shown in~\cite{MR988510}.
The equivalence of 3.\ and 4.\ was conjectured by Hall and Paige~\cite{MR0079589},
and proved by Bray, Evans and Wilcox using the classification of finite simple groups (see \cite{MR2469351}).
Trivially, 1.\ implies 2.
Also,  2.\ implies 3., as per the discussion before Conjecture~\ref{Brualdi-Stein}~\end{proof}.

\begin{corollary}
\label{cor: admissible groups}
We have
\begin{enumerate}
\item
 $\Cchi(L_G)=|G|$ for every group $G$ of odd order {\rm (}also see~{\rm\cite{MR1130611})}.
\item
For every  group $G$ of order $n$,  either $\Cchi(L_G) = n$ or $\Cchi(L_G) \ge n+2$.
\item
Let $G$ be an Abelian group of order $n$,
 $\Cchi(L_G)\ge n+2$ if and only if $G$ has a unique element of order $2$ {\rm (}also see~{\rm\cite{MR0020990})}.
 \end{enumerate}
 \end{corollary}
The rest of this section is devoted to proving that the lower bound of Corollary~\ref{cor: admissible groups}(3)
 is tight for cyclic groups.

\begin{theorem} \label{thm: cyclic chromatic number}
For $n\ge1$ the circulant Latin square of order $n$ satisfies
\[
\Cchi(L_{{\Bbb Z}_n}) = \begin{cases}
			n     & \text{if $n$ is odd}\\
			n+2 & \text{if $n$ is even.}
			\end{cases}
\]
\end{theorem}

Our proof of Theorem~\ref{thm: cyclic chromatic number} utilizes the following class of graphs.
\begin{definition}
The {\sf  M\"{o}bius ladder of order $2n$} is the cubic graph $M$ obtained from a cycle $C$ of length $2n$ by adding $n$ new edges, each connecting an opposite pair of vertices of $C$. The cycle $C$ is called the {\sf rim} of $M$, and the added edges are called the {\sf rungs} of $M$.
{\rm(}See Figure~{\rm \ref{fig:Z8Mobius6}} and Figure~{\rm \ref{fig:Z8Mobius}(b))}.
Two vertices of $M$ are said to be {\sf nearly antipodal} in $M$ if they are at distance $n-1$ in $C$.
\end{definition}
\begin{figure}[ht]
\begin{center}
\adjustbox{scale=.9}{

} 
\includegraphics[scale=.70]{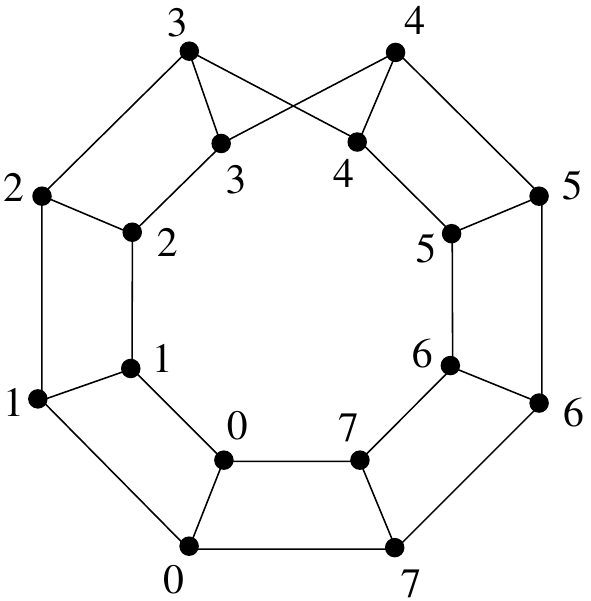}
\caption{{A M\"obius ladder of order 16.}}
\label{fig:Z8Mobius6}
\end{center}
\end{figure}
The following proposition is  straightforward to verify.
\begin{proposition}
\label{prop:bipartite}
If $x$, $x'$ are nearly antipodal vertices of a M\"obius ladder $M$, then $M-\{x,x'\}$ is bipartite.
\end{proposition}
%



Let $L_{{\Bbb Z}_{n}} = \{  (r,c,s) \in {\Bbb Z}_n \times {\Bbb Z}_n \times {\Bbb Z}_n \mid s = r+c \pmod n \}$ be the circulant Latin square of order $n$.
For $0 \le i < n$, we define the \textsf{$i$th right diagonal} of $L_{{\Bbb Z}_{n}}$ to be
the set  $T_i = \{(r, r+i, 2r+i) \      |   \ 0 \leq r \leq {n}-1 \}$.
\begin{proposition}\label{prop:moeb}
For $0 \le i < n$, the subgraph $\langle T_{i} \cup T_{i+1} \rangle$ of $\Gamma(L_{{\Bbb Z}_{n}})$ induced by the $2$-plex $ T_{i} \cup T_{i+1}$
is isomorphic to the  M\"{o}bius ladder of order $2n$.
Moreover if $n=2m$ is even, then every cell $x \in T_{i}$  is nearly antipodal to each of the cells $x' = x + (m, m+1, 1)$ and $x'' = x + (m-1, m, -1)$.
\end{proposition}
\begin{proof}
The row edges and column edges in $\langle T_{i} \cup T_{i+1} \rangle $ comprise the rim of the M\"obius ladder.
For each $(r,c,s) \in T_i$, we have $c-r=i$, $c+r = s$,
 and the only other solution to the system $\{ c'-r' \in \{i, i+1\}, c'+r'=s\} \pmod{n}$ is
$(r',c') = (r+\lfloor \frac{n}{2} \rfloor, c+\lceil \frac{n}{2} \rceil)$.
Therefore each symbol edge in $\langle T_{i} \cup T_{i+1} \rangle $ is a rung
joining an opposite pair of vertices along the rim of $\langle T_{i} \cup T_{i+1} \rangle$.
If $n=2m$ and $x \in T_i$, then
both $x'$ and $x''$ are nearly antipodal to $x$ since they are adjacent to $x+(m,m,0)$ along the rim of $\langle T_{i} \cup T_{i+1} \rangle$,
and $\{ x, x+(m,m,0)\}$ is a rung of $\langle T_{i} \cup T_{i+1} \rangle$.
\end{proof}\\
See Figure~\ref{fig:Z8Mobius} for an example.

\begin{figure}[ht]
\begin{center}
\adjustbox{scale=.9}{
\begin{tabular}
{|@{\hspace{3.1pt}}c@{\hspace{3.1pt}}|@{\hspace{3.1pt}}c@{\hspace{3.1pt}}|@{\hspace{3.1pt}}c@{\hspace{3.1pt}}|@{\hspace{3.1pt}}c@{\hspace{3.1pt}}|@{\hspace{3.1pt}}c@{\hspace{3.1pt}}|@{\hspace{3.1pt}}c@{\hspace{3.1pt}}|@{\hspace{3.1pt}}c@{\hspace{3.1pt}}|@{\hspace{3.1pt}}c@{\hspace{3.1pt}}|}
\hline
0     &  1     &   2     &   3      &    \textcolor{blue}{\large \bf \underline 4}      &   {\large \bf 5} &   6      &    7  \\ \hline
1     &  2     &   3     &  4      &     5      &   {\large \bf 6} &   {\large \bf 7}  &    0  \\ \hline
2     &  3     &   4     &   5      &     6      &    7     &    {\large \bf  0} &    {\large \bf 1}  \\  \hline
\textcolor{blue}{\large \bf \underline 3}    &  4     &   5     &   6      &     7     &    0     &   1      &   {\large \bf 2}  \\ \hline
 {\large \bf 4}&   {\large \bf 5}&   6     &    7      &     0      &    1     &   2      &    3  \\ \hline
5     &  {\large \bf 6}&    {\large \bf 7}&   0      &     1      &    2     &   3      &    4  \\ \hline
6     &  7     &    {\large \bf 0}&   {\large \bf 1} &     2      &    3     &   4      &    5  \\ \hline
7     &  0     &   1     &    {\large \bf 2} &     {\large \bf 3} &    4     &   5      &    6   \\ \hline
\end{tabular}
} 
\hspace{20mm}
\raisebox{-20mm}{\includegraphics[scale=.65]{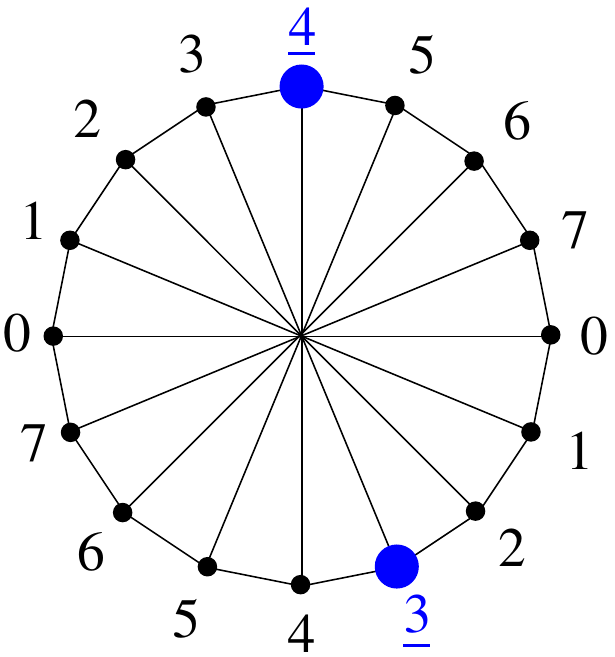}}\\
\smallskip
\hspace*{-1mm}(a) \hspace{54  mm} (b)
\caption{{For $L_{\Bbb{Z}_8}$ we indicate  $T_4 \cup T_5$ in bold face, and underline the near antipodal cells $x=(0,4,4)$ and $x''=x+(3,4,-1)=(3,0,3)$ in the M\"obius ladder $\langle T_4 \cup T_5 \rangle$.}}
\label{fig:Z8Mobius}
\end{center}
\end{figure}

\smallskip
\noindent
\textbf{Proof of Theorem \ref{thm: cyclic chromatic number}.}
If $n$ is odd it follows by Corollary~\ref{cor: admissible groups}(1). \\
If $n$ is even, by Corollary~\ref{cor: admissible groups}(3), it suffices to show that
$\Gamma(L_{{\Bbb Z}_{2m}})$ has a proper $(2m+2)$-coloring for $m\ge 1$.
This is true for $m=1$ since $\Gamma(L_{{\Bbb Z}_{2}}) \cong K_4$.
We assume $m\ge 2$ and let $k  = \lceil \frac {m}{2} \rceil$.
Using the notation $[t] = \{0, 1, \dots, t-1\}$, we define four sets of cells:
\begin{align*}
X &= \{ x_j \mid j \in [k]\},      & X' &= \{ x'_j \mid  j \in [k] \},             \\
Y &= \{ y_j \mid  j \in [m-k]\}, & Y' &= \{ y'_j \mid  j \in [m-k] \},
\end{align*}
where
\begin{align*}
\qquad\qquad x_j    &= 	(j,      3j     ,4j    ),               \qquad  & x'_j &= x_j + (d_r, d_c, d_s) , \quad\text{for $ j \in [k]$},     \\	
\qquad\qquad y_j    &= x_j + (0,2k,2k),                                    & y'_j &= y_j+(d_r, d_c, d_s) ,    \,\quad\text{for $ j \in [m-k]$},
\end{align*}
and where
\begin{equation*}
(d_r, d_c, d_s) = \begin{cases}
		(m,\,m+1,\,1) & \text{if $m \equiv 0 \pmod 3$,} \\
		(m-1,m,-1) & \text{if $m \not\equiv 0 \pmod 3$.}
	\end{cases}
\end{equation*}
Each of these  sets is (essentially)  a translation of $X$:
\begin{align}
X' &= X+(d_r, d_c, d_s) , \label{eq:X'}\\
Y \cup Y'  &= \begin{cases} 	(X \cup X') + (0,2k,2k) 			& \text{if $m$ is even,}	\\
					(X \cup X' \setminus \{x_k, x'_k\}) + (0,2k,2k) 	& \text{if $m$ is odd.}
		\end{cases}
		 \label{eq:YY'}\
\end{align}
See Figure~\ref{fig: Z12color} for examples.

\begin{figure}[h]
\begin{center}
\newcommand{\bfn}[1]{{\bf\large {\textcolor{red}{#1}}}}
\newcommand{\bfu}[1]{{\bf\large\textcolor{blue}{\underline {#1}}}}
\def\arraystretch{1.05}
\adjustbox{scale=.75}{
\begin{tabular}
{|@{\hspace{2pt}}c@{\hspace{2pt}}|@{\hspace{2pt}}c@{\hspace{2pt}}|@{\hspace{2pt}}c@{\hspace{2pt}}|@{\hspace{2pt}}c@{\hspace{2pt}}|@{\hspace{0pt}}c@{\hspace{0pt}}|@{\hspace{2pt}}c@{\hspace{2pt}}|@{\hspace{1pt}}c@{\hspace{1pt}}|@{\hspace{2pt}}c@{\hspace{2pt}}|@{\hspace{2pt}}c@{\hspace{2pt}}|@{\hspace{0pt}}c@{\hspace{0pt}}|@{\hspace{2pt}}c@{\hspace{2pt}}|@{\hspace{2pt}}c@{\hspace{1pt}}|}
  \hline
\bfn 0   & 1          & 2         &3              & 4           & 5    &\bfu 6  & 7         & 8        & 9           & 10     & 11\\ \hline
1          & 2           &3        &\bfn 4 & 5          & 6     & 7        & 8          & 9       & \bfu{10} & 11     & 0\\ \hline
\bfu 2   & 3          &4         & 5              & 6          & 7     &\bfn 8  & 9          & 10      & 11          & 0      & 1  \\  \hline
3          & 4           &5         & 6           & 7             & 8    & 9        & 10       & 11       & 0           & 1       & 2\\  \hline
4          & 5         &6         & 7              & 8           & 9    & 10       & 11       & 0          &  1         & 2        & 3 \\  \hline
5        & 6           &7         &8            & 9            & 10    & 11       & 0         & 1         &  2           & 3        & 4 \\  \hline
6        &\bfu 7    & 8         &9            &10           & 11   &  0         &\bfn 1   & 2         & 3            & 4        & 5 \\  \hline
7         & 8          &9         &10          & \bfu{11}  & 0     & 1         &2            &3         & 4           &\bfn 5   & 6 \\ \hline
8       &\bfn 9     &10       &11           & 0            & 1     & 2          &\bfu 3    & 4        & 5            & 6        & 7 \\  \hline
9        & 10          &11        &0           & 1           & 2     & 3          & 4           &5        & 6            & 7         & 8 \\  \hline
10      & 11           & 0         &1           & 2          & 3     & 4         & 5           & 6         & 7           & 8        & 9 \\  \hline
11        &  0        & 1         & 2            & 3          & 4      & 5         & 6          & 7         & 8           & 9        & 10 \\  \hline
\end{tabular}
\qquad\qquad
\begin{tabular}
{|@{\hspace{2pt}}c@{\hspace{1pt}}|@{\hspace{2pt}}c@{\hspace{2pt}}|@{\hspace{2pt}}c@{\hspace{2pt}}|@{\hspace{2pt}}c@{\hspace{2pt}}|@{\hspace{0pt}}c@{\hspace{0pt}}|@{\hspace{2pt}}c@{\hspace{2pt}}|@{\hspace{2pt}}c@{\hspace{2pt}}|@{\hspace{0pt}}c@{\hspace{0pt}}|@{\hspace{2pt}}c@{\hspace{2pt}}|@{\hspace{0pt}}c@{\hspace{0pt}}|@{\hspace{2pt}}c@{\hspace{2pt}}|@{\hspace{0pt}}c@{\hspace{0pt}}|@{\hspace{2pt}}c@{\hspace{2pt}}|@{\hspace{2pt}}c@{\hspace{2pt}}|}
  \hline
\bfn 0  &1           & 2      &3             & 4            & 5      &6           & 7           &\bfu 8    & 9            & 10         & 11          & 12      & 13   \\ \hline
1          &2          & 3      & \bfn4   & 5            & 6      &7          & 8            &9           & 10           &11         & \bfu{12}  & 13     & 0  \\ \hline
\bfu 2   & 3         & 4      &5              & 6           & 7      &\bfn 8   & 9            &10         & 11           & 12        & 13           & 0      & 1  \\  \hline
3          &4          & 5      &6              & 7           & 8      &9           &10           &11        & \bfn{12}   &13        & 0             &1       & 2\\  \hline
4          & 5         & 6      &7             & 8            & 9      &10        & 11           &12        & 13            & 0          &1           &2        & 3 \\  \hline
5          & 6         & 7      &8             & 9             & 10    &11       & 12          & 13         & 0            & 1           & 2          &3        & 4 \\  \hline
6          &\bfu 7   & 8       &9            &10            & 11    &12        &\bfn{13}  & 0            & 1          & 2           & 3           &4        & 5 \\  \hline
7           & 8        & 9       &10          &\bfu{11}   & 12     &13       & 0            & 1           & 2            &\bfn 3    & 4          &5        & 6  \\ \hline
8           & 9        & 10     &11          &  12          & 13    & 0         &\bfu 1      & 2           & 3           & 4           & 5          &6       &\bfn 7\\  \hline
9           & 10      &\bfn{11}& 12         & 13          & 0      & 1        & 2             & 3          & 4           & 5            & 6          &7        & 8\\  \hline
10         & 11       &12      & 13         & 0            & 1      & 2         & 3             &4          & 5           & 6            & 7          &8        & 9\\  \hline
11         & 12      & 13      & 0           & 1            & 2       & 3         & 4            &5          & 6            & 7           & 8           &9        & 10\\  \hline
12         & 13       & 0       & 1           & 2            & 3      & 4          & 5           &6           & 7            & 8           & 9          &10       & 11\\  \hline
13        &  0         & 1      & 2           & 3             & 4     & 5          & 6             &7          & 8             & 9          & 10         &11      & 12\\  \hline
\end{tabular}
} 
\caption{{The  cells $X \cup X'$ (\textcolor{red}{bold face}) and $Y \cup Y'$ (\textcolor{blue}{underlined bold face}) are shown for  $L_{\Bbb{Z}_{12}}$  and for $L_{\Bbb{Z}_{14}}$.}}
\label{fig: Z12color}
\end{center}
\end{figure}

\noindent
\textbf{Claim:}
Each of $X \cup X'$ and $Y \cup Y'$ is an independent set in the graph $\Gamma(L_{\Bbb{Z}_{2m}})$.

\noindent
\textbf{Proof of Claim:}
 By \eqref{eq:YY'} it suffices to prove the statement for
 $X \cup X' = \{ x_j, x'_j \mid  j \in [k] \}$.
 For any integer $t $, we define $\underline t \in [2m]$ to be the least positive residue of $t$ modulo $2m$.
 For  $j \in [k]$ we write $x_j = (r_j,c_j,s_j)$ and $x'_j = (r'_j,c'_j,s'_j)$, and define the following six multisets:
\begin{align*}
R &= \{\underline{r_0}, \underline{r_1}, \dots, \underline{r_{k-1}}\}, 	&	R' &= \{\underline{r'_0}, \underline{r'_1}, \dots, \underline{r'_{k-1}}\},	\\
C &= \{\underline{c_0}, \underline{c_1}, \dots, \underline{c_{k-1}}\},	&	C' &= \{\underline{c'_0}, \underline{c'_1}, \dots, \underline{c'_{k-1}}\},	\\
S &= \{\underline{s_0}, \underline{s_1}, \dots, \underline{s_{k-1}}\},	&	S' &= \{\underline{s'_0}, \underline{s'_1}, \dots, \underline{s'_{k-1}}\}.
\end{align*}
We aim to show that none of $ R \cup R'$, $C \cup C'$, $S \cup S'$ has a repeated entry.
For  $ j \in [k]$
each of the numbers $r_j = j$, $c_j = 3j$, $s_j = 4j$ is bounded above by $4(k-1) < 2m$, so we have
$$
\underline {r_j} = r_j = j, \qquad 
\underline {c_j}=c_j=3j,  \qquad 
\underline {s_j}= s_j=4j. 
$$
Therefore none of the lists $R$, $C$, $S$ has a repeated entry.
By \eqref{eq:X'}, none of the lists $R'$, $C'$, $S'$  has a repeated entry,
so it suffices to show that each of the sets $R \cap R'$, $C \cap C'$, $S \cap S'$ is empty.
First, since $m\ge2$ we have $k<m$ so $R \subseteq \{0,1,\dots,m-2\}$.
For $ j \in [k]$ we have $r'_j = r_j+d_r \in \{ r_j+m, r_j+(m-1) \} \subseteq \{m-1, m,\dots,2m-1\}$.
Therefore $\underline {r'_j}=r'_j$ and \\ $R \cap R' = \emptyset$.
Second, for $ j \in [k]$ we have $c'_j = c_j+d_c \in \{c_j+(m+1), c_j+m\}$.
Therefore $0 \le c'_j  < 4m$ and
$$
\underline {c'_j} = 
\begin{cases}
		 3j + d_c, &	\text{if $c_j+d_c < 2m $} \\
		 3j + d_c  -2m, &	\text{if $ c_j+d_c \ge 2m$}.
		 	\end{cases}
$$
 If $m \equiv 0 \pmod{3}$, then $d_c =m +1 $ so $\underline {c'_j} \equiv  1 \pmod{3}$.
If $m \not\equiv 0  \pmod{3}$, then $d_c = m $ so $\underline {c'_j} \equiv  \pm m \pmod{3}$.
In both cases we find $\underline {c'_j} \not\equiv 0 \equiv \underline {c_j} \pmod{3}$ for $ j \in [k]$.
Therefore $C \cap C' = \emptyset$.
%
Third, for $ j \in [k]$ we have $s'_j = s_j \pm1$ so $s_j$ is even and $s'_j$ is odd.
Since $2m$ is even, we conclude that $\underline{s_j}$ is even and $\underline{s'_j}$ is odd for every $ j \in [k]$.
Therefore $S \cap S'=\emptyset$ and the claim is proved.

\smallskip

For every $j \in [m]$ we apply Proposition \ref{prop:moeb} to the $2$-plex
$$
M_j := T_{2j} \cup T_{2j+1},
$$
 where $T_i$ is the $i$th right diagonal of  $L_{{\Bbb Z}_{2m}}$.
Each induced subgraph $\langle M_{j} \rangle$ is a M\"obius ladder.
For  $j \in [k]$ we  have that $x_j \in T_{2j}$,
so the cell $x'_j = x_j+(d_r, d_c, d_s)$ is nearly antipodal to $x_j$ in  $\langle M_{j} \rangle$.
Similarly, for  $j \in [m-k]$ we find that $y_j$ and $y'_j$ are nearly antipodal vertices of $\langle M_{j+k} \rangle$.
We now consider the following partition of  $L_{{\Bbb Z}_{2m}}$ into $m+1$ parts.
$$
 \{  M_{j} -\{x_j,x'_j\} \mid   j \in [k] \} \cup \{  M_{j+k} -\{y_j,y'_j\} \mid   j \in [m-k]  \}  \cup \{X \cup X' \cup Y \cup Y'\}
$$
By Proposition \ref{prop:bipartite} and the above claim, each part
induces a bipartite subgraph of~$\Gamma(L_{{\Bbb Z}_{2m}})$.
We conclude that $\Cchi(L_{{\Bbb Z}_{2m}}) \le 2m+2$.
\hspace{\fill}$\rule{1.15ex}{1.15ex}$\\
See Figure~\ref{fig: Z6color} for some examples.

%

\begin{figure}[h]
\begin{center}

\newcommand{\rn}[1] {\textcolor{red} {\textbf #1}}
\newcommand{\bn}[1]{\textcolor{blue}{\textbf #1}}
\def\arraystretch{1.0}

\adjustbox{scale=.90}{
\begin{tabular}
{|@{\hspace{2pt}}c@{\hspace{1.5pt}}|@{\hspace{1.5pt}}c@{\hspace{1.5pt}}|@{\hspace{1.5pt}}c@{\hspace{1.5pt}}|@{\hspace{1.5pt}}c@{\hspace{1.5pt}}|}\hline
\m{0}{\rn{4}}  & \m{1}{1}     &\m{2}{\bn5}   &\m{3}{3} \\ \hline
\m{1}{\bn5}    & \m{2}{0}     &\m{3}{\rn4}   & \m{0}{2} \\ \hline
\m{2}{3}         & \m{3}{2}     &\m{0}{1}       &\m{1}{0} \\\hline
\m{3}{0}         & \m{0}{3}     &\m{1}{2}       & \m{2}{1}  \\\hline
\end{tabular}
\qquad
\begin{tabular}
{|@{\hspace{1.5pt}}c@{\hspace{1.5pt}}|@{\hspace{1.5pt}}c@{\hspace{1.5pt}}|@{\hspace{1.5pt}}c@{\hspace{1.5pt}}|@{\hspace{1.5pt}}c@{\hspace{1.5pt}}|@{\hspace{1.5pt}}c@{\hspace{1.5pt}}|@{\hspace{1.5pt}}c@{\hspace{1.5pt}}|}
\hline
\m{0}{\rn6} & \m{1}{1}     & \m{2}{3}      & \m{3}{2}       & \m{4}{\bn7} & \m{5}{5} \\ \hline
\m{1}{5}     & \m{2}{0}     & \m{3}{1}      & \m{4}{\rn 6}  & \m{5}{3}      & \m{0}{4} \\ \hline
\m{2}{4}     & \m{3}{5}     & \m{4}{0}      & \m{5}{1}       & \m{0}{2}      & \m{1}{3} \\ \hline
\m{3}{3}     & \m{4}{4}     & \m{5}{\bn7} & \m{0}{0}       & \m{1}{\rn6}  & \m{2}{2} \\ \hline
\m{4}{2}     & \m{5}{\rn6} & \m{0}{5}      & \m{1}{4}       & \m{2}{1}      & \m{3}{0} \\ \hline
\m{5}{0}     & \m{0}{3}     & \m{1}{2}      & \m{2}{5}       & \m{3}{4}      & \m{4}{1} \\ \hline
\end{tabular}
\qquad
\begin{tabular}
{|@{\hspace{1.5pt}}c@{\hspace{1.5pt}}|@{\hspace{1.5pt}}c@{\hspace{1.5pt}}|@{\hspace{1.5pt}}c@{\hspace{1.5pt}}|@{\hspace{1.5pt}}c@{\hspace{1.5pt}}|@{\hspace{1.5pt}}c@{\hspace{1.5pt}}|@{\hspace{1.5pt}}c@{\hspace{1.5pt}}|@{\hspace{1.5pt}}c@{\hspace{1.5pt}}|@{\hspace{1.5pt}}c@{\hspace{1.5pt}}|}
\hline
\m{0}{\rn8} & \m{1}{1}     & \m{2}{2}      & \m{3}{3}       & \m{4}{\bn9} & \m{5}{5}    & \m{6}{6} & \m{7}{7}     \\ \hline
\m{1}{6}     & \m{2}{0}     & \m{3}{1}      & \m{4}{\rn 8}  & \m{5}{2}      & \m{6}{4}    & \m{7}{5} & \m{0}{\bn9} \\ \hline
\m{2}{7}     & \m{3}{6}     & \m{4}{0}      & \m{5}{1}       & \m{6}{3}      & \m{7}{2}    & \m{0}{4} & \m{1}{5}     \\ \hline
\m{3}{\bn9}& \m{4}{7}     & \m{5}{6}      & \m{6}{0}       & \m{7}{\rn8}  & \m{0}{3}    & \m{1}{2} & \m{2}{4}     \\ \hline
\m{4}{5}     & \m{5}{4}     & \m{6}{7}      & \m{7}{\bn9}  & \m{0}{1}      & \m{1}{0}    & \m{2}{3} & \m{3}{\rn8}  \\ \hline
\m{5}{3}     & \m{6}{5}     & \m{7}{4}      & \m{0}{6}       & \m{1}{7}      & \m{2}{1}    & \m{3}{0} & \m{4}{2}     \\ \hline
\m{6}{2}     & \m{7}{3}     & \m{0}{5}      & \m{1}{4}       & \m{2}{6}      & \m{3}{7}    & \m{4}{1} & \m{5}{0}     \\ \hline
\m{7}{0}     & \m{0}{2}     & \m{1}{3}      & \m{2}{5}       & \m{3}{4}      & \m{4}{6}    & \m{5}{7} & \m{6}{1}     \\ \hline
\end{tabular}
}
\caption{{ Colorings of $L_{\Bbb{Z}_{4}}$, $L_{\Bbb{Z}_{6}}$, and $L_{\Bbb{Z}_{8}}$ by the method of Theorem~\ref{thm: cyclic chromatic number}. The colors are given as subscripts.}}
\label{fig: Z6color}
\end{center}
\end{figure}

\section{Latin squares of small orders}
We have computed $\Cchi({\bf L})$ for all main classes of  order $n$,  $2 \le n \le 7$.
N.~Shajari has verified by computer that all 283657 main classes of Latin squares of order $8$ as listed in \cite{MR2291523} are 10-colorable.
 The results are summarized in Table \ref{tbl: chromatic_number}.

\def\arraystretch{1}
\begin{center}
\begin{tabular}{|c|c|l|c|}
\hline
$n$  & {Number of main classes}   & {Main class number} as listed in \cite{MR2246267}    		& $\Cchi$ \\ \hline\hline
2    				& 1                          		& 2.1  ($ L_{\Bbb{Z}_2}$)                         			& 4      \\ \hline\hline
3    				& 1                          		& 3.1  ($ L_{\Bbb{Z}_3}$)                       			& 3      \\ \hline\hline
\multirow{2}{*}{4} 	& \multirow{2}{*}{2} 		& 4.2  ($L_{\Bbb{Z}_2 \times\Bbb{Z}_2}$) 			& 4     \\\cline{3-4}
      				&                           		& 4.1  ($ L_{\Bbb{Z}_4}$)                 	    			& 6     \\ \hline\hline
\multirow{2}{*}{5}  	& \multirow{2}{*}{2}		& 5.1      ($L_{\Bbb{Z}_5}$)                    			& 5    \\ \cline{3-4}
     				&                             		& 5.2                                                    			& 6     \\ \hline\hline
\multirow{2}{*}{6}    	&\multirow{2}{*}{12}		& 6.2, 6.3, 6.4, 6.5, 6.10, 6.11              			& 7    \\  \cline{3-4}
     				&                            		& 6.1, 6.6, 6.7, 6.8, 6.9, 6.12                   			& 8     \\ \hline\hline
\multirow{2}{*}{7}    	& \multirow{2}{*}{147}  	& 7.3, 7.6, 7.7, 7.71, 7.105, 7.137       			& 7     \\ \cline{3-4}
     				&                            		&  All other main classes of order 7     			& 8     \\ \hline\hline
8   				& 283657                  		&All main classes as listed in \cite{MR2291523}	& $\le 10$   \\  \hline

\end{tabular}
\vspace*{-5mm}
\begin{table}[ht]
\caption{{ Chromatic numbers of Latin squares of order $n$,  $2 \le n \le 8$.}}
\label{tbl: chromatic_number}
\end{table}
\end{center}
%
%
%
%
%


 \noindent
{\bf Acknowledgements.}
The work of  Luis Goddyn was supported by  NSERC Canada,  while  E.S. Mahmoodian was visiting Simon Fraser University and  was partially supported by INSF.We also are grateful to the referees for their constructive input.  \\ 

 \noindent
{\bf Addendum.}
One of the referees for this paper asserts that  $\Cchi( \mathbf L )$ has been studied before by at least four different groups or individuals, he/she has discussed the invariant with some of these groups, and has seen it mentioned in several conference talks, but knows of no written work on the topic.
Another referee informs us that some of our results are contained in a paper by Nicholas Cavenagh and Jaromy Kuhl that is
the same concept but they call it ``chromatic index'' rather
than ``chromatic number'', and is
about to be published in \emph{Contributions to Discrete Mathematics} 12
(2016), issue 2.\\ \vspace*{5mm}

\noindent
{\bf\large{References}}

In the references, each reference is followed by a number which is
the page on which the reference is cited.
\def\cprime{$'$}


 \vspace*{5mm}

\noindent
{\tt nbesharati@pnu.ac.ir,
 goddyn@sfu.ca,
 emahmood@sharif.edu, \\
m.mortezaeefar@gmail.com}

\end{document}